\documentclass[12pt]{article}
\usepackage{amsmath}
\usepackage{graphicx,psfrag,epsf}
\usepackage{psfrag,epsf,booktabs}
\usepackage{enumerate}
\usepackage{psfrag,epsf}
\usepackage{url} 
\usepackage{amsmath,amssymb,amsfonts,amsthm,mathtools,mathrsfs}
\usepackage[table]{xcolor}

\usepackage{bm,bbm}
\usepackage{color}
\usepackage{subfigure}
\usepackage{algorithm,algpseudocode}
\usepackage[symbol]{footmisc}
\usepackage{scalefnt}
\usepackage{authblk} 
\usepackage{multirow,centernot}
\usepackage[colorlinks=true, citecolor=blue, urlcolor=blue]{hyperref}
\usepackage{natbib}
\usepackage{dsfont}
\usepackage[title]{appendix}
\usepackage[left=1in,top=1.1in,right=0.5in,bottom=1in]{geometry}

\theoremstyle{definition}

\newtheorem{theorem}{Theorem}[section]

\newtheorem{corollary}[theorem]{Corollary}

\newtheorem{proposition}[theorem]{Proposition}
\newtheorem{remark}[theorem]{Remark}

\makeatletter
\def\@seccntformat#1{\@ifundefined{#1@cntformat}%
	{\csname the#1\endcsname\quad}
	{\csname #1@cntformat\endcsname}
}
\makeatother

\newif\ifShowComments
\ShowCommentstrue
\def\strutdepth{\dp\strutbox}
\def\druk#1{\strut\vadjust{\kern-\strutdepth
        {\vtop to \strutdepth{%
                \baselineskip\strutdepth\vss
                        \llap{\hbox{#1}\quad}\null}}}}

\title{\bf
On unbiased estimators for functions of the rate parameter of the exponential distribution
}

\author{
\text{Roberto Vila}$^{1}$\thanks{Corresponding author: Roberto Vila, email: {rovig161@gmail.com}
}
\,\,\,and
\text{Eduardo Yoshio Nakano}$^{1}$
\\
{\small $^{1}$ Department of Statistics, University of Brasilia, Brasilia, Brazil}\\
}
\begin{document}
	\maketitle 	
	\begin{abstract}
%
In this paper, we explicitly derive unbiased estimators for various functions of the rate parameter of the exponential distribution {\color{black} in the absence of a location parameter}, including powers of the rate parameter, the 
$q$th quantile, the $p$th moment, the survival function, the maximum, minimum, probability density function, mean past lifetime, moment generating function, and others. 
{\color{black}
This work non-trivially complements established formulas for unbiased estimators of functions of parameters of the location-rate exponential distribution.
}
%
%
Additionally, we establish a result demonstrating the asymptotic normality of the proposed unbiased estimators.
	\end{abstract}
%

\section{Introduction}	

	

An estimator is said to be unbiased when its expected value equals the population parameter it is intended to estimate. The use of unbiased estimators is fundamental in statistical inference, as it ensures, on average, the accuracy of estimates and prevents systematic bias in the results. Among the most commonly employed methods for constructing estimators are the method of moments and maximum likelihood method. The former is particularly advantageous for directly producing estimators of linear transformations of the parameters, while the latter offers greater generality, enabling the derivation of estimators for arbitrary bijective transformations. Moreover, it is well known that when the underlying distribution is a member of the exponential family, the moment estimators and the maximum likelihood estimators (MLEs) coincide \citep{Davidson1974}.

Two well-established properties of MLEs are their invariance under bijective transformations and their asymptotic unbiasedness. These properties facilitate the construction of asymptotically unbiased estimators for any bijective transformation of the parameter of interest. However, it is well-known that the maximum likelihood (ML) method can produce biased estimators. Therefore, even if the MLE is unbiased for a parameter $\lambda$, the MLE of $\xi(\lambda)$ may be biased for a transformation $\xi(\lambda)$, unless $\xi(\cdot)$ is linear. Therefore, MLEs may not be ideal when the objective is to obtain unbiased estimators of specific functions of the parameter, especially when the sample size is not large enough.

To address this limitation, several authors have proposed alternative methods for deriving unbiased estimators. \cite{Washio1956} introduced a technique based on the Laplace transform to obtain unbiased estimators of transformations within the exponential family. \cite{Tate1959} extended this approach using the Laplace, bilateral Laplace, and Mellin transforms to derive unbiased estimators for distributions involving location and scale parameters. Following this same idea, \cite{Seheult1971} presented formulations of unbiased estimators for some probability density functions.

{\color{black}
It is important to note that the estimators proposed in this work are derived from exponential populations without a location parameter. As such, they complement the general formula introduced by \cite{Tate1959} for unbiased estimators of functions of parameters in the location–rate exponential distribution. Within this framework, the aim of this study is to provide unbiased estimators for various transformations of the rate parameter that were not addressed in previous research.
}
%
%
Several examples of unbiased estimators for functions of interest in the exponential distribution are provided, along with a result that proof the asymptotic normality of the proposed unbiased estimators.

\section{Preliminaries and the main result}

Let $X_1,\ldots, X_n$ be a random sample of size $n$ from $X\sim\exp(\lambda)$ (exponential distribution with rate parameter $\lambda>0$) and let $\xi(\lambda)$ be a populational characteristic of $X$, for some borel measurable function $\xi:(0,\infty)\to\mathbb{R}$. 
Some examples of functions $\xi(\lambda)$ are given in Table \ref{table:1}. This work focuses on deriving an unbiased estimator for $\xi(\lambda)$, specifically a random variable $h(X_1,\ldots,X_n)$ where
\begin{align*}
	\mathbb{E}\left[h(X_1,\ldots,X_n)\right]
	=
	\xi(\lambda),
\end{align*}
for some measurable function $h:(0,\infty)^n\to\mathbb{R}$.


As $\mathbb{E}(X)=1/\lambda$, it is natural to consider
\begin{align*}
	h(X_1,\ldots,X_n)=\phi(\overline{X}),
\end{align*}
where $\overline{X}=(1/n)\sum_{i=1}^{n}X_i$ is the sample mean, for some borel measurable function $\phi:(0,\infty)\to\mathbb{R}$. Therefore, the problem is reduced to obtaining the function $\phi$ in the following equation:
\begin{align}\label{eq-1}
	\mathbb{E}\left[\phi(\overline{X})\right]
	=
	\xi(\lambda).
\end{align}

\smallskip 


The following theorem presents the main result of this work, providing an explicit formula for unbiased estimators $\phi(\overline{X})$ 
for the populational parameter $\xi(\lambda)$.

\begin{theorem}\label{main theorem}
	If $X_1,\ldots,X_n$ is a random sample of size $n$ from $X\sim\exp(\lambda)$, then, the following sample function
	\begin{align}\label{final-exp}
		\phi(\overline{X})
		=
		\int_0^{\infty} 
		\mathds{1}_{\{\overline{X}\geqslant v\}}
		\left(1-{v\over \overline{X}}\right)^{n-1} 
		\mathscr{L}^{-1}\left\{\xi\left({s\over n}\right)
		\right\}(v){\rm d}v
	\end{align}
	is an unbiased estimator for $\xi(\lambda)$, where $F(s)=\mathscr{L}\left\{f(x)\right\}(s)=\int_0^\infty f(x)\exp(-sx){\rm d}x$ denotes the Laplace transform and $f(x)=\mathscr{L}^{-1}\left\{F(s)\right\}(x)$ is its respective inverse.  In the above we are assuming that the inverse of the Laplace transform and the respective improper integral exist.
\end{theorem}

\begin{proof}
	It is well-known that  $\overline{X}\sim\text{Gamma}(n,n\lambda)$. Then the equality \eqref{eq-1} can be written as
	\begin{align*}
		{(n\lambda)^n\over \Gamma(n)}
		\int_{0}^{\infty}
		\phi(x)
		x^{n-1}
		\exp(-n\lambda x)
		{\rm d}x
		=
		\xi(\lambda),
	\end{align*}
	which can be expressed in function of Laplace transform as follows
	\begin{align*}
		\mathscr{L}\left\{	\phi(x)
		x^{n-1}\right\}(n\lambda)
		=
		\Gamma(n)\,
		{\displaystyle \xi\left({1\over n} \, (n\lambda)\right)\over (n\lambda)^n}
		=
		\Gamma(n)\,
		{\displaystyle \xi\left({s\over n}\right)\over s^n}
		\Bigg\vert_{s=n\lambda}.
	\end{align*}
	Applying the inverse Laplace transform (if it exists) to both sides gives
	\begin{align}\label{main-identit-0}
		\phi(x)
		x^{n-1}
		=
		\Gamma(n)\,
		\mathscr{L}^{-1}\left\{	
		{\displaystyle \xi\left({s\over n}\right)\over s^n}
		\right\}
		(x)
		=
		\Gamma(n)\,
		\mathscr{L}^{-1}\left\{	
		F(s)H(s)
		\right\}
		(x),
	\end{align}	
	where we are adopting the notation $F(s)=1/s^n$ and $H(s)=\xi\left({s/n}\right)$. 
	
	Since
	\begin{align*}
		f(x)=  \mathscr{L}^{-1}\left\{	F(s)
		\right\}(x)
		=
		{x^{n-1}\over\Gamma(n)}
		\quad
		\text{and}
		\quad 
		h(x)=\mathscr{L}^{-1}\left\{H(s)
		\right\}(x)
		=
		\mathscr{L}^{-1}\left\{\xi\left({s\over n}\right)
		\right\}(x),
	\end{align*}
	by applying the convolution theorem for the Laplace transform, we obtain
	\begin{align}
		\mathscr{L}^{-1}\left\{	
		F(s) H(s)
		\right\}
		(x)
		=
		(f*h)(x)
		&=
		\int_0^x f(x-v)h(v){\rm d}v
		\nonumber
		\\[0,2cm]
		&=
		{1\over\Gamma(n)}
		\int_0^x (x-v)^{n-1} 
		\mathscr{L}^{-1}\left\{\xi\left({s\over n}\right)
		\right\}(v){\rm d}v
		\label{ide-3-3-1}
		\\[0,2cm]
		&=
		{x^{n-1} \over\Gamma(n)}
		\int_0^\infty 
		\mathds{1}_{\{x\geqslant v\}}
		\left(1-{v\over x}\right)^{n-1} 
		\mathscr{L}^{-1}\left\{\xi\left({s\over n}\right)
		\right\}(v){\rm d}v, \label{ide-3-3}
	\end{align}
	where $x\mapsto(f*h)(x)$ denotes  the convolution function. By plugging \eqref{ide-3-3} in \eqref{main-identit-0}, we get
	\begin{align}
		\phi(x)
		=
		\int_0^\infty 
		\mathds{1}_{\{x\geqslant v\}}
		\left(1-{v\over x}\right)^{n-1} 
		\mathscr{L}^{-1}\left\{\xi\left({s\over n}\right)
		\right\}(v){\rm d}v.
	\end{align}
	Thus, the unbiased estimator $\phi(\overline{X})$ takes the form given in \eqref{final-exp}, which completes the proof of the theorem.
\end{proof}

\begin{corollary}\label{coro-main}
	Under the conditions of Theorem \ref{main theorem}, we have 
	\begin{align*}
		\phi(\overline{X})
		=
		{\Gamma(n)\over \overline{X}^{n-1}}\,
		\mathscr{L}^{-1}\left\{	
		{\displaystyle \xi\left({s\over n}\right)\over s^n}
		\right\}
		(\overline{X}),
	\end{align*}
	where $\mathscr{L}^{-1}\left\{F(s)\right\}(x)$ is the inverse Laplace  transform.
\end{corollary}
\begin{proof}
	The proof follows directly from \eqref{main-identit-0} by replacing $x$ with $\overline{X}$.
\end{proof}

\begin{remark}
	Since $\overline{X}\sim{\rm Gamma}(n,n\lambda)$, by using formula in Corollary \ref{coro-main}, we have
	\begin{align*}
		\mathbb{E}[\phi(\overline{X})]
		&=
		{(n\lambda)^n}
		\int_0^\infty 
		\mathscr{L}^{-1}\left\{	
		{\displaystyle \xi\left({s\over n}\right)\over s^n}
		\right\}
		(x) \
		\exp(-n\lambda x)
		{\rm d}x
		\\[0,2cm]
		&=
		{(n\lambda)^n} \,
		\mathscr{L}\left\{
		\mathscr{L}^{-1}\left\{	
		{\displaystyle \xi\left({s\over n}\right)\over s^n}
		\right\}
		(x)
		\right\}(n\lambda)
		\\[0,2cm]
		&=
		\xi(\lambda).
	\end{align*}
	This reinforces the unbiasedness of the estimator $\phi(\overline{X})$ of Corollary \ref{coro-main} and Theorem \ref{main theorem}.
\end{remark}

	\begin{remark}
		It is important to note that {\color{black} when $X_1,\ldots,X_n$ is a random sample from $X$ having a location-rate exponential distribution, 
        denoted by $X\sim\text{exp}(\theta,\lambda)$, 
        that is, its probability density function is given by 
$
f_X(x;\theta,\lambda)=\lambda \exp\{-\lambda(x-\theta)\}, x>\theta,
$
\cite{Tate1959} (Item 7.4) proposed the following unbiased  estimator for the function $\xi(\theta,\lambda)$: 
\begin{align*}
	\phi(X_{(1)},\overline{X}-X_{(1)})
=
{\Gamma(n-1)\over [\overline{X}-X_{(1)}]^{n-2}} \, 
\mathscr{L}^{-1}\left\{
{\xi(\theta,s/n)\over s^{n-1}} \, 
-	{{\partial\xi\over\partial\theta}(\theta,s/n)\over s^{n}} 
\right\}(\overline{X}-X_{(1)}) \, 
\Bigg\vert_{\theta=X_{(1)}},
\end{align*}
where $X_{(1)}\equiv\min\{X_1,\ldots,X_n\}$.
Although $X\sim\text{exp}(\theta=0,\lambda)=\text{exp}(\lambda)$, the above estimator $\phi(X_{(1)},\overline{X}-X_{(1)})$ cannot be reduced to the form of the estimator $\phi(\overline{X})$ given in Corollary \ref{coro-main}. Thus, the results presented in this work provide a non-trivial complement to the unbiased estimator formulas established by \cite{Tate1959}.
}
	\end{remark}

\begin{corollary}
	Under the conditions of Theorem \ref{main theorem}, we have 
	\begin{align*}
		\phi(\overline{X})
		=
		{\Gamma(n)\over \overline{X}^{n-1}}\,
		\mathcal{R}_n\left\{\mathscr{L}^{-1}\left\{\xi\left({s\over n}\right)
		\right\}(v)\right\}(\overline{X}),
	\end{align*}
	where $\mathcal{R}_\mu\left\{f(v)\right\}(x)=[1/\Gamma(\mu)]\int_0^x (x-v)^{\mu-1}f(v){\rm d}v$ is the Riemann-Liouville fractional integral and $\mathscr{L}^{-1}\left\{F(s)\right\}(x)$ is the inverse Laplace  transform.
\end{corollary}
\begin{proof}
	The proof follows imediately by combining \eqref{ide-3-3-1} with \eqref{main-identit-0}.
\end{proof}

\section{Examples of unbiased estimators}

In what follows we apply Formula \eqref{final-exp} in Theorem \ref{main theorem} to obtain some unbiased estimators for specific population parameters $\xi(\lambda)$ of $X\sim\exp(\lambda)$.

\subsection{Rate parameter power}\label{Rate parameter power}

For  $X\sim\exp(\lambda)$, consider the following populational characteristic
\begin{align*}
	\xi(\lambda)
	=
	\lambda^p,
\end{align*}
where $p$ is given. 

By the linearity of the inverse Laplace transform, we have
\begin{align*}
	\mathscr{L}^{-1}\left\{	
	{\displaystyle \xi\left({s\over n}\right)}
	\right\}
	(x)
	=
	{1\over n^p } \,
	\mathscr{L}^{-1}&\left\{	
	{\displaystyle 	
		s^{p}
	}
	\right\}
	(x)
	=
	{1\over n^p \Gamma(-p)} \,
	{x^{-p-1}},
\end{align*}
because $\mathscr{L}^{-1}\left\{	
s^{p}
\right\}(x)=x^{-p-1}/\Gamma(-p)$.
Hence, from Theorem \ref{main theorem}, an unbiased estimator for the power of rate parameter of  $X$ is given by
\begin{align}
	\phi(\overline{X})
	&=
	\int_0^{\infty} 
	\mathds{1}_{\{\overline{X}\geqslant v\}}
	\left(1-{v\over \overline{X}}\right)^{n-1} 
	\mathscr{L}^{-1}\left\{\xi\left({s\over n}\right)
	\right\}(v){\rm d}v
	\nonumber
	\\[0,2cm]
	&=
	{1\over n^p \Gamma(-p)} 
	\int_0^{\infty} 
	\mathds{1}_{\{\overline{X}\geqslant v\}}
	\left(1-{v\over \overline{X}}\right)^{n-1} 
	{v^{-p-1}}
	{\rm d}v
	\nonumber
	\\[0,2cm]
	&=
	{{\rm B}(-p,n)\over n^p \Gamma(-p)}  \, \overline{X}^{-p}
	\nonumber
	\\[0,2cm]
	&=
	{\Gamma(n)\over n^p\Gamma(n-p)}  \, \overline{X}^{-p},
    \quad p<n,
	\label{formula-power}
\end{align}
because $\int_0^x (1-v/x)^{n-1} v^{-p-1}{\rm d}v={\rm B}(-p,n)x^{-p}$. In the above, ${\rm B}(x,y)=\Gamma(x)\Gamma(y)/\Gamma(x+y)$ is the beta function.

\subsection{Quantile}\label{q-quantile}

For  $X\sim\exp(\lambda)$,  the   $q$th quantile is given by
\begin{align*}
	\xi(\lambda)={-{\dfrac {\ln(1-q)}{\lambda }}} \, \mathds{1}_{\{0<q<1\}}.
\end{align*}

By using the unbiased estimator for the rate parameter power (see Subsection \ref{Rate parameter power})  with $p=-1$, it is clear that an  unbiased estimator for the $q$th quantile of $X$ is given by.
\begin{align*}
	\phi(\overline{X})= -\ln(1-q) \overline{X} \, \mathds{1}_{\{0<q<1\}}.
\end{align*}

\subsection{Moment}

For  $X\sim\exp(\lambda)$,  the $p$th  moment, with $p>-1$ known, is given by
\begin{align*}
	\xi(\lambda)=\mathbb{E}[X^p]={\Gamma(p+1)\over \lambda^p}.
\end{align*}

Taking $-p$ instead of $p$ in Formula \eqref{formula-power}, we have that
\begin{align*}
	\phi(\overline{X})
	=
	{\Gamma(p+1)\Gamma(n) n^p\over \Gamma(p+n)} \, \overline{X}^{p}
\end{align*}
is an unbiased estimator for the $p$th real moment of $X$.

\subsection{Survival function}
\label{Unbiased estimator for the survival function}

For $X\sim\exp(\lambda)$, the survival function is given by
\begin{align*}
	\xi(\lambda)=\mathbb{P}(X>t)=\exp(-\lambda t),
\end{align*}
where $t>0$ is known. 

Hence,
\begin{align*}
	\mathscr{L}^{-1}\left\{	
	{\displaystyle \xi\left({s\over n}\right)}
	\right\}
	(x)
	=
	\mathscr{L}^{-1}\left\{	
	{\displaystyle
		\exp\left(-{t\over n} \, s\right)}
	\right\}(x) 
	=
	\delta\left(x-{t\over n}\right),
\end{align*}
because $\mathscr{L}^{-1}\left\{\exp(-as)\right\}(x)=\delta(x-a)$, where $\delta(x)$ is the Dirac delta function. Therefore, from Theorem \ref{main theorem}, an unbiased estimator for the survival function of  $X$ is given by
\begin{align*}
	\phi(\overline{X})
	&=
	\int_0^{\infty} 
	\mathds{1}_{\{\overline{X}\geqslant v\}}
	\left(1-{v\over \overline{X}}\right)^{n-1} 
	\mathscr{L}^{-1}\left\{\xi\left({s\over n}\right)
	\right\}(v){\rm d}v
	\\[0,2cm]
	&=
	\int_0^{\infty} 
	\mathds{1}_{\{\overline{X}\geqslant v\}}
	\left(1-{v\over \overline{X}}\right)^{n-1} 
	\delta\left(v-{t\over n}\right)
	{\rm d}v
	\\[0,2cm]
	&=
	{\left(1 - {t\over n \overline{X}}\right)^{n-1}} \, \mathds{1}_{\{\overline{X}\geqslant {t\over n}\}},
\end{align*}
where in the last line we have used the sifting property (or the sampling property)  of Dirac delta function:  $\int _{-\infty }^{\infty }f(t)\,\delta (t-T){\rm d}t=f(T)$.

\subsection{Maximum}\label{Unbiased estimator for the maximum}

Let $X_1,\ldots,X_m$ be $m$ independent copies of $X\sim\exp(\lambda)$. The cumulative distribution function (CDF) of $\max\{X_1,\ldots,X_m\}$ is
\begin{align*}
	\xi(\lambda)=[\mathbb{P}(X\leqslant t)]^m=[1-\exp(-\lambda t)]^m,
\end{align*}
where $t>0$ is known and $m\in\mathbb{N}$. 

By Newton binomial theorem the parameter $\xi(\lambda)$ can be written as
\begin{align*}
	\xi(\lambda)=1+\sum_{k=1}^m\binom{m}{k}(-1)^k\exp(-\lambda k t).
\end{align*}
Then, by using the unbiased estimator for the survival function of $X$, given in Subsection \ref{Unbiased estimator for the survival function}, it is clear that an unbiased estimator for $[\mathbb{P}(X\leqslant t)]^m$ is given by
\begin{align*}
	\phi(\overline{X})
	=
	1
	+
	\sum_{k=1}^m\binom{m}{k}(-1)^k 	{\left(1 - {kt\over n \overline{X}}\right)^{n-1}} \, \mathds{1}_{\{\overline{X}\geqslant {kt\over n}\}}. 
\end{align*}

\subsection{Minimum}

Let $X_1,\ldots,X_m$ be $m$ independent copies of $X\sim\exp(\lambda)$. The survival function of $\min\{X_1,\ldots,X_m\}$ is
\begin{align*}
	\xi(\lambda)=[\mathbb{P}(X> t)]^m=\exp(-\lambda m t),
\end{align*}
where $t>0$ is known and $m\in\mathbb{N}$. 

By using the unbiased estimator for the survival function of $X$, given in Subsection \ref{Unbiased estimator for the survival function}, it is clear that an unbiased estimator for $[\mathbb{P}(X> t)]^m$ is given by
\begin{align*}
	\phi(\overline{X})
	=
	{\left(1 - {mt\over n \overline{X}}\right)^{n-1}} \, \mathds{1}_{\{\overline{X}\geqslant {mt\over n}\}}. 
\end{align*}

\subsection{Probability density function}

For $X\sim\exp(\lambda)$, the probability density function (PDF) is defined by
\begin{align*}
	\xi(\lambda)=\lambda\exp(-\lambda t),
\end{align*}
where $t>0$ is known. 

Using the linearity property of the inverse Laplace transform, we obtain
\begin{align*}
	\mathscr{L}^{-1}\left\{	
	{\displaystyle \xi\left({s\over n}\right)}
	\right\}
	(x)
	=
	{1\over n}\,
	\mathscr{L}^{-1}\left\{	
	{\displaystyle
		s
		\exp\left(-{t\over n} \, s\right)}
	\right\}(x) 
	=
	{1\over n}\,
	\delta'\left(x-{t\over n}\right),
\end{align*}
because $\mathscr{L}^{-1}\left\{s\exp(-as)\right\}(x)=\delta'\left(x-a\right)$, where $\delta'(x)$ is the distributional derivative of the Dirac delta function.
Thus, from Theorem \ref{main theorem}, an unbiased estimator for the PDF of  $X$ is given by
\begin{align*}
	\phi(\overline{X})
	&=
	\int_0^{\infty} 
	\mathds{1}_{\{\overline{X}\geqslant v\}}
	\left(1-{v\over \overline{X}}\right)^{n-1} 
	\mathscr{L}^{-1}\left\{\xi\left({s\over n}\right)
	\right\}(v){\rm d}v
	\\[0,2cm]
	&=
	{1\over n}
	\int_0^{\infty} 
	\mathds{1}_{\{\overline{X}\geqslant v\}}
	\left(1-{v\over \overline{X}}\right)^{n-1} 
	\delta'\left(v-{t\over n}\right) 
	{\rm d}v
	\\[0,2cm]
	&=
	\left({n-1\over n}\right) {1\over\overline{X}}
	{\left(1 - {t\over n \overline{X}}\right)^{n-2}} \, \mathds{1}_{\{\overline{X}\geqslant {t\over n}\}},
\end{align*}
where in the last equality we have used the sifting property of $\delta'(x)$.

\subsection{Mean past lifetime}

For  $X\sim\exp(\lambda)$, the mean past lifetime can be expressed as
\begin{align*}
	\xi(\lambda)=\mathbb{E}[t-X\vert X\leqslant t]
	=
	{t\over 1-\exp(-\lambda t)}-{1\over\lambda},
\end{align*}
where $t>0$ is given. 

By the linearity of the inverse Laplace transform, we have
\begin{align}\label{def-prop-1}
	\mathscr{L}^{-1}\left\{	
	{\displaystyle \xi\left({s\over n}\right)}
	\right\}
	(x)
	&=
	t
	\mathscr{L}^{-1}\left\{	
	{1\over 1-\exp(-{t\over n} \, s)}
	\right\}
	(x)
	- 
	n
	\mathscr{L}^{-1}\left\{{1\over s}
	\right\}
	(x)
	\nonumber
	\\[0,2cm]
	&=
	t
	\sum_{k=0}^\infty 
	\delta\left(x-{tk\over n}\right)
	- 
	n,
\end{align}
because $\mathscr{L}^{-1}\{1/s\}(x)=1$ and 
$\mathscr{L}^{-1}\{{1/[1-\exp(-as)]}\}(x)=\sum_{k=0}^\infty \delta\left(x-{a k}\right)$,  where $\delta(x)$ is the Dirac delta function.
Thus, from Theorem \ref{main theorem}, an unbiased estimator for the mean past lifetime of  $X$ is given by
\begin{align*}
	\phi(\overline{X})
	&=
	\int_0^{\infty} 
	\mathds{1}_{\{\overline{X}\geqslant v\}}
	\left(1-{v\over \overline{X}}\right)^{n-1} 
	\mathscr{L}^{-1}\left\{\xi\left({s\over n}\right)
	\right\}(v){\rm d}v
	\\[0,2cm]
	&=
	t
	\sum_{k=0}^\infty 
	\int_0^{\infty} 
	\mathds{1}_{\{\overline{X}\geqslant v\}}
	\left(1-{v\over \overline{X}}\right)^{n-1} 
	\delta\left(v-{tk\over n}\right)
	{\rm d}v   
	-
	n
	\int_0^{\infty} 
	\mathds{1}_{\{\overline{X}\geqslant v\}}
	\left(1-{v\over \overline{X}}\right)^{n-1} 
	{\rm d}v  
	\\[0,2cm]
	&=
	t
	\sum_{k=0}^\infty 
	\left(1-{tk\over n \overline{X}}\right)^{n-1}
	\mathds{1}_{\{\overline{X}\geqslant {tk\over n}\}}
	-
	\overline{X},
\end{align*}
where the sifting property of the Dirac delta function was applied in the final equality.

\subsection{Moment generating function}
\label{Unbiased estimator for the moment generating function}

For  $X\sim\exp(\lambda)$, the moment generating function (MGF) is written as
\begin{align*}
	\xi(\lambda)=\mathbb{E}[\exp(tX)]
	=
	{\frac{\lambda}{\lambda -t}},
\end{align*}
where $t<\lambda$ is given. 

Note that
\begin{align*}
	\mathscr{L}^{-1}\left\{	
	{\displaystyle \xi\left({s\over n}\right)}
	\right\}
	(x)
	&=
	\mathscr{L}^{-1}\left\{
	{\frac{s}{s -nt}}
	\right\}
	(x)
	=
	nt\exp(ntx)+\delta(x),
\end{align*}
because $\mathscr{L}^{-1}\left\{s/(s-a)\right\}(x)=a\exp(ax)+\delta(x)$,  where $\delta(x)$ is the Dirac delta function. Therefore, from Theorem \ref{main theorem}, an unbiased estimator for the MGF of $X$ can be written as
\begin{align*}
	\phi(\overline{X})
	&=
	\int_0^{\infty} 
	\mathds{1}_{\{\overline{X}\geqslant v\}}
	\left(1-{v\over \overline{X}}\right)^{n-1} 
	\mathscr{L}^{-1}\left\{\xi\left({s\over n}\right)
	\right\}(v){\rm d}v
	\\[0,2cm]
	&=
	nt
	\int_0^{\infty} 
	\mathds{1}_{\{\overline{X}\geqslant v\}}
	\left(1-{v\over \overline{X}}\right)^{n-1} 
	\exp(ntv)
	{\rm d}v
	+
	\int_0^{\infty} 
	\mathds{1}_{\{\overline{X}\geqslant v\}}
	\left(1-{v\over \overline{X}}\right)^{n-1} 
	\delta(v)
	{\rm d}v
	\\[0,2cm]
	&=
	{\exp(nt\overline{X}) \gamma(n,nt\overline{X})\over (nt\overline{X})^{n-1}}
	+
	1,
\end{align*}
because $    \int_0^{x} 
\left(1-{v/x}\right)^{n-1} 
\exp(ntv)
{\rm d}v
=
x\exp(ntx) \gamma(n,ntx)/(ntx)^n
$, where $\gamma (s,x)$ is the lower incomplete gamma function: ${\int _{0}^{x}t^{s-1}\exp({-t}){\rm d}t}$.
Note that in the last equality above the sifting property of the Dirac delta function was also used.

\smallskip

All unbiased estimators of the population parameters $\xi(\lambda)$ of $X\sim\exp(\lambda)$ given in Subsections  \ref{Rate parameter power}-\ref{Unbiased estimator for the moment generating function}
are summarized in the Table \ref{table:1}.
{\small
\begin{table}[H]
	\centering
	\caption{Unbiased estimators for some population parameters $\xi(\lambda)$ of $X\sim\exp(\lambda)$.}
	\label{table:1}
\resizebox{\linewidth}{!}{
	\begin{tabular}[t]{lcc}
		\toprule
		Population parameter & $\xi(\lambda)$ & Unbiased 
		estimator $\phi(\overline{X})$  \\
		\midrule
		\rowcolor{gray!10} 
		Rate parameter power
		& $\lambda^p$  &  $\dfrac{\Gamma(n)}{n^p\Gamma(n-p)}  \, \overline{X}^{-p} \, \mathds{1}_{\{p<n\}}$
		\\[3.0ex] 
		$q$th quantile
		&  ${-{\dfrac {\ln(1-q)}{\lambda }}} \, \mathds{1}_{\{0<q<1\}}
		$  &  $-\log(1-q) \overline{X} \, \mathds{1}_{\{0<q<1\}}
		$
		\\[3.0ex] \rowcolor{gray!10} 
		$p$th  moment 
		&  $\dfrac{\Gamma(p+1)}{\lambda^p}\, \mathds{1}_{\{p>-1\}}$  & $	
		\dfrac{\Gamma(p+1)\Gamma(n) n^p}{\Gamma(p+n)} \, \overline{X}^{p}$   
		\\[3.0ex]
		Survival function
		&  $\exp(-\lambda t)$  & $	{\left(1 - \dfrac{t}{n \overline{X}}\right)^{n-1}} \, \mathds{1}_{\{\overline{X}\geqslant {t\over n}\}}$ 
		\\[3.0ex] \rowcolor{gray!10} 
		Maximum
		&  $ 
		[\mathbb{P}(X\leqslant t)]^m=[1-\exp(-\lambda t)]^m$  & 
		$  \displaystyle 
		1
		+
		\sum_{k=1}^m\binom{m}{k}(-1)^k 	{\left(1 - {kt\over n \overline{X}}\right)^{n-1}} 
		\mathds{1}_{\{\overline{X}\geqslant {kt\over n}\}}$ 
		\\[3.0ex]
		Minimum
		&  $[\mathbb{P}(X> t)]^m=\exp(-\lambda tm)$  & $ \displaystyle 
		{\left(1 - {mt\over n \overline{X}}\right)^{n-1}} \, \mathds{1}_{\{\overline{X}\geqslant {mt\over n}\}}$ 
		\\[3.0ex]   \rowcolor{gray!10} 
		Probability density function & $\lambda\exp(-\lambda t)$ &
		$ \displaystyle 
		\left({n-1\over n}\right) {1\over\overline{X}}
		{\left(1 - {t\over n \overline{X}}\right)^{n-2}} \, \mathds{1}_{\{\overline{X}\geqslant {t\over n}\}}$
		\\[3.0ex]   
		Mean past lifetime
		&  $\dfrac{t}{1-\exp(-\lambda t)}-\dfrac{1}{\lambda}$  & $    
		t \displaystyle 
		\sum_{k=0}^\infty 
		\left(1-\dfrac{tk}{n \overline{X}}\right)^{n-1}
		\mathds{1}_{\{\overline{X}\geqslant {tk\over n}\}}
		-
		\overline{X}
		$     
		\\[3.0ex] \rowcolor{gray!10} 
		Moment generating function
		&
		$\displaystyle 
		{\frac{\lambda}{\lambda -t}}\, \mathds{1}_{\{t<\lambda\}}$
		&     $\displaystyle 
		{\exp(nt\overline{X}) \gamma(n,nt\overline{X})\over (nt\overline{X})^{n-1}}
		+
		1 $ 
		\\[3.0ex]		
		Expected shortfall	
		& $\dfrac{-\ln(1-p)+1}{\lambda} \, \mathds{1}_{\{0<p<1\}}
		$ & $[{-\ln(1-p)+1}]\overline{X} \, \mathds{1}_{\{0<p<1\}}$  
		\\[3.0ex]
		\bottomrule
	\end{tabular}
}
\end{table}
}

\begin{remark}
	It is clear that from the $p$th moment estimator in Table 	\ref{table:1} we can obtain estimators for the mean and variance.
	Additionally, estimators for the hazard function, cumulative hazard function, residual mean lifetime and Fisher information follow directly from the rate parameter power estimator.
\end{remark}

The next result shows that our unbiased estimator $\widehat{\phi}$ for the $p$th moment (see Table \ref{table:1}) outperforms the correspondent MLE estimator in terms of efficiency.
\begin{proposition}\label{prop-compare-var}
	Let $\widehat{\psi}$ be the unbiased maximum likelihood estimator (MLE) for the $p$th moment of $X\sim\exp(\lambda)$ (which coincides with the moment estimator \citep{Davidson1974}), 
	that is,
	\begin{align*}
		\widehat{\psi}\equiv{1\over n}\sum_{i=1}^n X_i^p,
	\end{align*} 
	and let
	\begin{align*}
		\widehat{\phi}
		\equiv 
		\phi(\overline{X})
		=
		\dfrac{\Gamma(p+1)\Gamma(n) n^p}{\Gamma(p+n)} \, \overline{X}^{p}
	\end{align*}
	be the proposed estimator for  the $p$th moment (see Table \ref{table:1}).
	We have:
	\begin{align*}
		{\rm Var}(\widehat{\phi})<{\rm Var}(\widehat{\psi}),
	\end{align*}
	for all $p>-1/2$ and nonzero.
\end{proposition}
\begin{proof} 
	As $X\sim \exp(\lambda)$, we have
	\begin{align*}
		{\rm Var}(\widehat{\psi})
		=
		{1\over n}\, {\rm Var}(X^p)
		=
		{\Gamma^2(p+1)\over\lambda^{2p}}
		\left[
		{\Gamma(2p+1)\over n\Gamma^2(p+1) }
		+
		\left(1-{1\over n}\right)
		-
		1
		\right],
		\quad p>-1/2.
	\end{align*}
	
	On the other hand, as $\overline{X}\sim\text{Gamma}(n,n\lambda)$, we get
	\begin{align*}
		{\rm Var}(\widehat{\phi})
		=
		\dfrac{\Gamma^2(p+1)\Gamma^2(n) n^{2p}}{\Gamma^2(p+n)} \,
		{\rm Var}(\overline{X}^{p})
		=
		\dfrac{\Gamma^2(p+1) }{\lambda^{2p}} \,
		\left[
		{\Gamma(n) \Gamma(2p+n)\over \Gamma^2(p+n) }
		-
		1
		\right],
		\quad p>-n/2.
	\end{align*}
	
	Notice that ${\rm Var}(Z_n)= {\rm Var}(\widehat{\phi})$ when $p=0$.  From now on we assume that $p\neq 0$.  We see that
	${\rm Var}(Z_n)> {\rm Var}(\widehat{\phi})$ if and only if
	\begin{align*}
		{\Gamma(2p+1)\over n\Gamma^2(p+1) }
		+
		\left(1-{1\over n}\right)
		>
		{\Gamma(n) \Gamma(2p+n)\over \Gamma^2(p+n) },
	\end{align*}
	which equivalently can be written as
	\begin{align*}
		{1\over n} \,
		{\Gamma(2p+1)\over \Gamma^2(p+1) }
		-
		\Gamma(n) \,
		{\Gamma(2p+n)\over \Gamma^2(p+n) }
		+
		1
		>
		{1\over n}.
	\end{align*}
	However, note that the above strict inequality always holds because ${\Gamma(2p+1)/\Gamma^2(p+1) }> 1$ and $ {\Gamma(2p+n)/\Gamma^2(p+n) }> 1$ for all $p>-1/2$ and nonzero.
	This concludes the proof.
\end{proof}

\begin{remark}
	Note that the estimators of the rate parameter power (with $p=1$) and the survival function in Table \ref{table:1} are Uniformly Minimum Variance Unbiased Estimators (UMVUE estimators) \citep[see Example 35, p. 328, in][]{Mood1913}.
\end{remark}

\section{Asymptotic behavior}

Applying the strong law of large numbers, we have
\begin{align*}
	\overline{X}\stackrel{\rm a.s.}{\longrightarrow} {1\over\lambda},
\end{align*}
where $``\stackrel{\rm a.s.}{\longrightarrow}"$ denotes almost sure convergence. Hence, continuous-mapping
theorem gives
\begin{align}\label{strong-convergence}
	\phi(\overline{X})\stackrel{\rm a.s.}{\longrightarrow}  \phi\left({1\over\lambda}\right),
\end{align}
whenever $\phi:(0,\infty)\to\mathbb{R}$ is defined by \eqref{final-exp}, that is,
\begin{align*} 
	\phi(x)
	=
	\int_0^{\infty} 
	\mathds{1}_{\{x\geqslant v\}}
	\left(1-{v\over x}\right)^{n-1} 
	\mathscr{L}^{-1}\left\{\xi\left({s\over n}\right)
	\right\}(v){\rm d}v.
\end{align*}

Furthermore,  by the 
central limit theorem (CLT), we have
\begin{align*}
	\sqrt{n}\left(\overline{X}-{1\over\lambda}\right)
	\stackrel{\mathscr{D}}{\longrightarrow} \text{N}\left(0,{1\over\lambda^2}\right),
\end{align*}
where $``\stackrel{\mathscr{D}}{\longrightarrow}"$ means convergence in distribution. 
If we additionally impose the condition that the sequence $\{\phi(\overline{X})\}_{n\geqslant 1}$ is uniformly integrable, then, by standard results  \citep[cf.][Theorem 5.4]{Billingsley1968}, from \eqref{strong-convergence} we would have that
\begin{align}\label{id-xi-phi}
	\xi(\lambda)
	=
	\lim_{n\to\infty}\mathbb{E}[\phi(\overline{X})]
	=
	\phi\left({1\over\lambda}\right),
\end{align}
where the first equality is valid by the fact that the estimator $\phi(\overline{X})$ is unbiased.
Thus, the delta method yields
\begin{align}\label{CLT}
	\sqrt{n}\left[\phi(\overline{X})-\xi(\lambda)\right]
	\stackrel{\eqref{id-xi-phi}}{=}
	\sqrt{n}\left[\phi(\overline{X})-\phi\left({1\over\lambda}\right)\right]
	\stackrel{\mathscr{D}}{\longrightarrow} N\left(0,\left[{1\over\lambda} \,  \phi'\left({1\over\lambda}\right)\right]^2\right),
\end{align}
whenever the derivative of $\phi(x)$ evaluated at the point $x=1/\lambda$, which is given by
\begin{align}\label{derivative-phi}
	\phi'\left({1\over\lambda}\right)
	=
	(n-1) \lambda^2
	\int_0^{\infty} 
	\mathds{1}_{\{1/\lambda \geqslant  v\}} \,
	v
	(1-\lambda{v})^{n-2} 
	\mathscr{L}^{-1}\left\{\xi\left({s\over n}\right)
	\right\}(v){\rm d}v,
\end{align}
exists and is non-zero valued.
Hence, by \eqref{strong-convergence}, we obtain the strong consistency of estimator $\phi(\overline{X})$, and from the  convergence \eqref{CLT}, a corresponding
CLT. The results found are summarized in the following theorem.
\begin{theorem}\label{theo-clt}
	If $X_1,\ldots,X_n$ is a random sample of size $n$ from $X\sim\exp(\lambda)$, then the unbiased estimator $\phi(\overline{X})$ of $\xi(\lambda)$, given in Theorem \ref{main theorem}, satisfies the following convergences:
	\begin{align*}
		\phi(\overline{X})\stackrel{\rm a.s.}{\longrightarrow}  \xi(\lambda)
	\end{align*}
	and
	\begin{align*}
		\sqrt{n}\left[\phi(\overline{X})-\xi(\lambda)\right]
		\stackrel{\mathscr{D}}{\longrightarrow} N\left(0,\left[{1\over\lambda} \,  \phi'\left({1\over\lambda}\right)\right]^2\right),
	\end{align*}
	provided the sequence $\{\phi(\overline{X})\}_{n\geqslant 1}$ is uniformly integrable and $\phi'\left({1/\lambda}\right)$ in \eqref{derivative-phi} exists and is non-zero valued.
\end{theorem}

\begin{remark}
	The sequence $\{\phi(\overline{X})\}_{n\geqslant 1}$, where
	$$\phi(\overline{X})=\dfrac{\Gamma(p+1)\Gamma(n) n^p}{\Gamma(p+n)} \, \overline{X}^{p}$$
	is the unbiased estimator of the $p$th  moment (see Table \ref{table:1}), is uniformly integrable for all $p>-1/2$ and nonzero.
	
	Indeed, to prove that the sequence $\{\phi(\overline{X})\}_{n\geqslant 1}$ is uniformly integrable, it suffices to prove the following uniform bound: for some $\varepsilon>0$, $\sup_{n\geqslant 1}\mathbb{E}[\vert\phi(\overline{X})\vert^{1+\varepsilon}]<\infty$, which is immediate, since by Proposition \ref{prop-compare-var}, we have 
	\begin{align*}
		\sup_{n\geqslant 1}
		\mathbb{E}[\{\phi(\overline{X})\}^{2}]
		=
		\sup_{n\geqslant 1}
		{\rm Var}(\phi(\overline{X}))
		\leqslant
		{\Gamma^2(p+1)\over\lambda^{2p}}
		\sup_{n\geqslant 1}
		\left[
		{\Gamma(2p+1)\over n\Gamma^2(p+1) }
		+
		\left(1-{1\over n}\right)
		-
		1
		\right]
		<\infty,
	\end{align*}
	for all $p>-1/2$ and nonzero. Consequently, the conditions of Theorem \ref{theo-clt} are satisfied.
\end{remark}

\section{Discussion}

In this paper, we explicitly derive unbiased estimators for various functions of the rate parameter of the exponential distribution using the Laplace transform. 
%
{\color{black}
The results presented herein complements unbiased estimators proposed by \cite{Tate1959} of functions of parameters of the location-rate exponential distribution. }


We provided a 
list of these estimators as well as several other unbiased estimators for population quantities of interest, including the $p$-th moment, survival function, minimum, probability density function \cite[previously presented by][]{Seheult1971}, mean past lifetime, moment generating function, and expected shortfall. It is worth noting that unbiased estimators for any other function of interest that is a linear combination of the quantities mentioned above can be readily obtained from the estimators provided, such as the variance, hazard function, cumulative hazard function, among others.

We also showed that, under certain existence and integrability conditions, the distribution of the proposed estimators is asymptotically normal. This result is particularly important, as it enables the construction of confidence intervals and the performance of hypothesis tests for large samples. Although the distributions of MLEs are also asymptotically normal, it is worth emphasizing that our estimators have the advantage of being unbiased by construction, whereas MLEs are not necessarily so. This advantage can be particularly relevant in scenarios that involve moderately large samples.

It is also important to mention that the proposed technique has limitations. In some cases, the inverse Laplace transform and the corresponding improper integral may not exist, which makes the analytical derivation of an unbiased estimator impossible using this methodology. Nevertheless, in such cases, the inverse (if it exists) can often be obtained numerically. In particular, we can compute the inverse Laplace transform in the R program using the \textit{invlaplace} function from the ``pracma" package \cite{Borchers2023}.

Future research may explore the derivation of unbiased estimators for other quantities of interest under different probability distributions. Although this topic has received little attention in recent decades, it may regain popularity due to current advances in computational methods. In situations where asymptotic normality cannot be guaranteed, either due to insufficient sample size and/or violation of the necessary conditions (existence and integrability for the proposed estimators, or regularity conditions in the case of MLEs), it is common to adopt computational techniques based on resampling. In such cases, the use of unbiased estimators, such as those proposed in this work, will help prevent biases in results, thus eliminating the need for subsequent bias corrections.


\paragraph*{Acknowledgements}
The research was supported in part by CNPq and CAPES grants from the Brazilian government.

\paragraph*{Disclosure statement}
There are no conflicts of interest to disclose.


\bibliographystyle{apalike}


\end{document}